\documentclass[reqno,a4paper,10pt]{amsart}
\usepackage [latin1]{inputenc}
\usepackage{amssymb}
\usepackage[english]{babel}
\usepackage[dvips, lmargin=2.9cm, rmargin=2.9cm, tmargin=3cm, bmargin=2.7cm]{geometry}
\usepackage[colorlinks=true,urlcolor=blue,linkcolor=blue]{hyperref}
\usepackage{hyperref}
\usepackage[T1]{fontenc}
\usepackage[english]{babel}
\usepackage{dsfont}
\usepackage{enumitem}
\usepackage{dsfont}
\usepackage{color}

\everymath{\displaystyle}

\newtheorem{thm}{Theorem}[section]
\newtheorem{defini}[thm]{Definition}
\newtheorem{rem}[thm]{Remark}
\newtheorem{lem}[thm]{Lemma}

\newcommand{\Bex}{\mathbf{B}_\epsilon(x)}

\newcommand{\bre}[1]{\left\{#1\right\}}

\newcommand{\n}[1]{\left\vert#1\right\vert}

\newcommand{\dx}{\hspace{0.2mm}dx}

\numberwithin{equation}{section}

\begin{document}
\title{\sf Continuity of weak solutions to an elliptic problem on $p$-fractional Laplacian
}
\author{\sl Wei Chen$^1$, Qi Han$^{2}$, Guoping Zhan$^{3*}$}
\address{School of Science, Chongqing University of Posts and Telecommunications,\vskip0pt Chongqing 400065, P.R. China\hspace{12mm}\textit{Email:}\hspace{2mm}{\tt weichensdu@126.com}}
\address{Department of Mathematics, Texas A\&M University, \vskip0pt San Antonio, Texas 78224, USA\hspace{12mm}\textit{Email:}\hspace{2mm}{\tt qhan@tamusa.edu}}
\address{Department of Mathematics, Zhejiang University of Technology,\vskip0pt Hangzhou 310023, Zhejiang,  P.R. China\hspace{12mm}\textit{Email:}\hspace{2mm}{\tt zhangp@zjut.edu.cn}}

\thanks{$^{*}$ Corresponding author: Guoping Zhan (E-mail address: zhangp@zjut.edu.cn)}

\maketitle

\textbf{Abstract}. In this paper we study an elliptic variational problem regarding the $p$-fractional Laplacian in $\mathbb{R}^N$ on the basis of recent result \cite{Ha1}, which generalizes the nice work \cite{AT,AP,XZR1}, and then give some sufficient conditions under which some weak solutions to the above elliptic variational problem are continuous in $\mathbb{R}^N$. In the final appendix we correct the proofs of both \cite[Lemma 10]{PXZ1} and  \cite[Lemma A.6]{PXZ} for $1<p<2$.

\section{Introduction}
We consider in this paper the following problem
\begin{equation}\label{problem1.1}
(-\Delta)_p^su+V(x)|u|^{p-2}u=\lambda a(x)|u|^{r-2}u-b(x)|u|^{q-2}u
\end{equation}
in $\mathbb{R}^N$, where up to a normalization constant one defines
\begin{equation*}
(-\Delta)_p^su(x):=2\lim_{\epsilon\to0^+}\int_{\mathbb{R}^N\setminus\Bex}\frac{|u(x)-u(y)|^{p-2}(u(x)-u(y))}{|x-y|^{N+ps}}dy
\end{equation*}
for $x\in\mathbb{R}^N$ with $\Bex:=\{y\in\mathbb{R}^N:|x-y|<\epsilon\}$, under the assumptions as below.
\begin{enumerate}
\item \label{C1} $N\geq2$, $\lambda>0$, $0<s<1$, and $1<p<r<\min\{q,p_s^*\}$, with $N>ps$ and $p_s^*:=\frac{Np}{N-ps}$.
\item \label{C2} $V:\mathbb{R}^N\rightarrow\mathbb{R}^+$ with a positive constant $V_0$ such that $V(x)\geq V_0>0$ for all $x\in\mathbb{R}^N$.
\item \label{C3} $a, b:\mathbb{R}^N\rightarrow\mathbb{R}^+$ satisfying
$$0<\int_{\mathbb{R}^N}a(x)^{\frac{p_s^*+\beta(p_s^*-p)+\gamma(p_s^*-q)}{p_s^*-r}}b(x)^{-\gamma}\dx<\infty$$ for some $\beta\in[0,\infty)$ and $\gamma\in\Big(0,\frac{(r-p)\big(p_s^*+\beta(p_s^*-p)\big)}{(q-r)(p_s^*-p)}\Big)$
when $a, b\in L^1_{loc}(\mathbb{R}^N)$, or some $\beta\in[0,\infty)$ and $\gamma=\frac{(r-p)\big(p_s^*+\beta(p_s^*-p)\big)}{(q-r)(p_s^*-p)}$
when $a\in L^k(\mathbb{R}^N)$ with $k>1$ and $b\in L^1_{loc}(\mathbb{R}^N)$.
\end{enumerate}
It is closely related to the following Dirichlet problem with indefinite weights:
\begin{eqnarray}\label{E1.2}
\left\{\begin{aligned}
-\Delta u-\lambda u &=\omega(x) u^{q-1}-h(x) u^{r-1} & & \text { in } \Omega \\
u(x) &>0 & & \text { in } \Omega \\
u(x) &=0 & & \text { on } \partial \Omega
\end{aligned}\right.
\end{eqnarray}
where $\lambda \in \mathbb{R}, \Omega \subset \mathbb{R}^{N}(N \geq 3)$ is a bounded domain with smooth boundary, the coefficients $\omega, h \in L^{\infty}(\Omega)$ are nonnegative and $2<q<r$. Alama and Tarantello \cite{AT} proved the existence, nonexistence and multiplicity of solutions to (\ref{E1.2}) depending on $\lambda$ and according to the integrability of the ratio $w^{r-2} / h^{q-2}$.
\par In \cite{PR}, Pucci and R$\check{\rm a}$dulescu considered the following related problem in the whole space:
\begin{eqnarray}\label{E1.3}
\left\{\begin{aligned}
-\operatorname{div}\left(|\nabla u|^{p-2} \nabla u\right)+|u|^{p-2}u &=\lambda |u|^{q-2}u-h(x) |u|^{r-2}u & & \text { in } \mathbb{R}^{N} \\
u(x) & \geq 0 & & \text { in } \mathbb{R}^{N}
\end{aligned}\right.
\end{eqnarray}
where $h>0$ satisfies $0<\int_{\mathbb{R}^{N}} h(x)^{q /(q-r)} d x<\infty$ and
$\lambda$ is a positive parameter and $2\leq p<q<r<p^{*}$, with $p^{*}=N p /(N-p)$ if $N>p$ and $p^{*}=\infty$ if $N \leq p.$ They showed the nonexistence and existence of nontrivial solutions for the smallness and the largeness of $\lambda$, respectively.

\par Later, Autuori and Pucci \cite{AP1} extended above (\ref{E1.3}) to the following quasilinear elliptic equation:
\begin{equation}\label{E1.4}
   -\operatorname{div} A(x, \nabla u)+a(x) |u|^{p-2} u=\lambda \omega(x) |u|^{q-2} u-h(x) |u|^{r-2} u \quad \text { in } \mathbb{R}^{N},
\end{equation}
where $A(x, \nabla u)$ acts like the $p$-Laplacian, $\max \{2, p\}<q<\min \left\{r, p^{*}\right\}$ with $p^{*}=N p /(N-p)$ and $1<p<N$ , the coefficients $\omega$ and $h$ are related by the integrability condition
\begin{equation*}
    \int_{\mathbb{R}^{N}}\left[\frac{\omega^{r}(x)}{h^{q}(x)}\right]^{1 /(r-q)} d x \in \mathbb{R}^{+}.
\end{equation*}

\par After that, Autuori and Pucci in \cite{AP} turned to the following elliptic equation involving the fractional Laplacian:
\begin{equation}\label{E1.5}
(-\Delta)^{s} u+a(x) u=\lambda \omega(x)|u|^{q-2} u-h(x)|u|^{r-2} u \quad \text { in } \mathbb{R}^{N},
\end{equation}
where $\lambda\in\mathbb{R},0<s<1, 2s<N$ and $2<q<\min \left\{r, 2_s^{*}\right\}$ with $2_s^{*}=2N /(N-2s)$,  $(-\Delta)^{s}$ is the fractional Laplacian operator. They studied the existence and multiplicity of entire solutions to (\ref{E1.5}) by using variational methods and the mountain pass theorem.

\par In \cite{PZ}, Pucci and Zhang considered the following quasilinear elliptic equations in the setting of variable exponents:
\begin{equation}\label{E1.6}
   -\operatorname{div}A(x, \nabla u)+a(x) |u|^{p(x)-2} u=\lambda \omega(x) |u|^{q(x)-2} u-h(x) |u|^{r(x)-2} u \quad \text { in } \mathbb{R}^{N}.
\end{equation}
In fact, they obtained the existence of entire solutions of (\ref{E1.6}), which generalized (\ref{E1.4}) from the case of constant exponents $p,q$ and $r$ to the case of variable exponents. They also extended the previous work of Alama and Tarantello \cite{AT} from Dirichlet Laplacian problems in bounded domains of $\mathbb{R}^N$ to the case of general variable exponent differential equation in the $\mathbb{R}^N$. Furthermore, they solved the two open problems proposed in \cite{AP}. Recently, Pucci et al. \cite{PXZ} also gave a positive answer to these open problems in the context of Kirchhoff problems involving the fractional $p$-Laplacian.

\par More recently, Xiang et al.\cite{XZR1} investigated the existence, nonexistence and multiplicity of nontrivial weak solutions to (\ref{problem1.1}) depending on $\lambda$ and according to the integrability of the ratio $a^{q-p}/b^{r-p}$ by using variational methods. In fact, they extended the results of Autuori and Pucci to the fractional  $p$-Laplacian and weakened the conditions in their paper.

\par Motivated by \cite{XZR1}, in this paper we will extend the well known results on existence and multiplicity of weak solutions of (\ref{problem1.1}). Furthermore, we would like to give some sufficient conditions under which the weak solutions of (\ref{problem1.1}) are continuous.
Finally, we correct the proofs of both \cite[Lemma 10]{PXZ1} and \cite[Lemma A.6]{PXZ} for $1<p<2$.
\par Now we give the definition of weak solutions of problem (\ref{problem1.1}).
\begin{defini} \label{solution}
 We say that $u\in W$ {\rm (}or $u\in W_M\subset W${\rm)} is a weak solution of problem $(\ref{problem1.1})$ if
\begin{eqnarray}
&&\langle (-\Delta)_p^su, \varphi\rangle+\int_{\mathbb{R}^{N}} V(x)|u(x)|^{p-2} u(x) \varphi(x)dx \nonumber \\
&=&\lambda \int_{\mathbb{R}^{N}} a(x)|u(x)|^{r-2} u(x)\varphi(x) d x-\int_{\mathbb{R}^{N}} b(x)|u(x)|^{q-2} u(x) \varphi(x)dx
\end{eqnarray}
for any $\varphi \in W$, where $$\langle(-\Delta)_p^su, \varphi\rangle=\iint_{\mathbb{R}^{2 N}} \frac{|u(x)-u(y)|^{p-2}(u(x)-u(y))(\varphi(x)-\varphi(y))}{|x-y|^{N+p s}}dxdy $$
and the solution spaces $W$ and $W_M$ will be introduced in Section $2$.
\end{defini}
Firstly, we shall generalize the condition and result of \cite[Theorem 1.1]{XZR1} on problem (\ref{problem1.1})
and prove the following theorem.
\begin{thm}\label{Th1}
Suppose that the three conditions {\rm (\ref{C1})-(\ref{C3})} are satisfied.
Then there exist a universal constant $\lambda_*>0$ and two constants $\lambda_M$ and $\lambda_M^*$ depending on $M$
with $\lambda_M^*\geq\lambda_M>0$ for each fixed $M>0$ such that problem \eqref{problem1.1} has
\par\noindent {\rm(i)} only the trivial weak solution in $W$ if $\lambda<\lambda_*$ when $a\in L^k(\mathbb{R}^N)$
with $k>1$ and $b\in L^1_{loc}(\mathbb{R}^N)$, \par and only the trivial weak solution in $W_M$ if $\lambda<\lambda_M$
when $a, b\in L^1_{loc}(\mathbb{R}^N)$;
\par\noindent {\rm (ii)} at least two nontrivial nonnegative weak solutions in $W_M$ in which one has negative energy
and \par another has positive energy if $\lambda>\lambda_M^*$. \end{thm}

\begin{rem} \rm {Since $W_M\subset W$ our result generalizes the Theorem 1.1 in \cite{XZR1}. And our condition (\ref{C3})
is new even in the regular Laplacian setting. More precisely, $L^{\frac{p_s^*}{p_s^*-r}}(\mathbb{R}^N)\subset L_{loc}^1(\mathbb{R}^N)$ implies
the condition (H3) in \cite{XZR1} is essentially $a\in L^{\frac{p_s^*}{p_s^*-r}}(\mathbb{R}^N)$, which corresponds to the case when
$a\in L^k(\mathbb{R}^N)$ with $k=\frac{p_s^*}{p_s^*-r}>1$ and $b\in L^1_{loc}(\mathbb{R}^N)$
in our condition (\ref{C3}); in this case if $\beta=0$ and
$\gamma=\frac{(r-p)\big(p_s^*+\beta(p_s^*-p)\big)}{(q-r)(p_s^*-p)}$, then $\frac{p_s^*+\beta(p_s^*-p)+\gamma(p_s^*-q)}{p_s^*-r}=\frac{q-p}{q-r}\cdot\frac{N}{ps}$
and $\gamma=\frac{r-p}{q-r}\cdot\frac{N}{ps}$, which implies the condition (H4) in \cite{XZR1}
is a special case of our condition (\ref{C3}), so our condition (\ref{C3}) generalizes the conditions (H3) and (H4) in \cite{XZR1}.
Moreover, our result also generalizes the nice work \cite{AT,AP} as well as the closely related ones \cite{Ha1,Ha2,PXZ,PZ,XZR2},
and can be extended to improve the work \cite{JL,LW,Qi} etc. among many others.}\end{rem}

\par Secondly, applying Theorem 3.1 and Theorem 3.13 in \cite{BP} we give some sufficient conditions under which
some weak solutions to the problem \eqref{problem1.1} are continuous in $\mathbb{R}^N$.
\begin{thm}\label{continuity} Suppose that the three conditions {\rm (\ref{C1})-(\ref{C3})} are satisfied and that
$a(x), b(x), V(x)\in L^{\infty}(\Omega)$ for any bounded domain $\Omega\subset\mathbb{R}^N$.
If $q\leq p_s^*$, then any weak solution in $W$ to problem \eqref{problem1.1} is continuous
in $\mathbb{R}^{N}$. If $q> p_s^*$, then any weak solution in $W\cap L_{loc}^{\frac{p_s^*(q-p)}{p_s^*-p}}(\mathbb{R}^{N})$
to problem \eqref{problem1.1} is continuous in $\mathbb{R}^{N}$.
\end{thm}

\begin{rem}  \rm {When $q\leq p_s^*$ and $u\in W$ is a weak solution with bounded support denoted by ${\rm supp}u$, this theorem follows
from applying the argument of \cite[Theorem 3.3]{CMS} with corresponding $\Omega={\rm supp}u$, $\alpha=0$ and $f(x,u)=\lambda a(x)|u|^{r-2}u-b(x)|u|^{q-2}u-V(x)|u|^{p-2}u$ since the condition (3.8) of \cite[Theorem 3.3]{CMS} holds a.e. in $\Omega$.
However, the condition (3.8) of \cite[Theorem 3.3]{CMS} is not applicable for our case when $q>p_s^*$ since our condition
is not included in (3.8) of \cite[Theorem 3.3]{CMS} in this case.}
\end{rem}

\section{Preliminaries}
In this section, we first give some basic results of fractional Sobolev spaces that will be used in the next section. Let $0<s<1<p<\infty$ be real numbers and the fractional Sobolev space $W^{s,p}(\mathbb{R}^N)$ be defined as follows:
\begin{equation*}
    W^{s,p}(\mathbb{R}^N)=\left\{u\in L^{p}(\mathbb{R}^N):[u]_{s,p}^p<\infty\right\}
\end{equation*}
equipped with the norm
\begin{equation*}
    ||u||_{W^{s,p}(\mathbb{R}^N)}=\left(||u||_{L^p(\mathbb{R}^N)}^p+[u]_{s,p}^p\right)^
{\frac{1}{p}},
\end{equation*}
where
\begin{equation*}
    [u]_{s,p}=\Big(\int\int_{\mathbb{R}^{2N}}\frac{|u(x)-u(y)|^p}{|x-y|^{N+ps}}dxdy\Big)^{\frac{1}{p}}.
\end{equation*}

By \cite[Theorem 6.7]{NPV} we know that the embedding $W^{s,p}(\mathbb{R}^N)\hookrightarrow L^{\nu}(\mathbb{R}^N)$
is continuous for any $\nu\in [p,p_s^*]$ with
a positive constant $C=C(N,p,s)$ such that \begin{equation}\label{uLv}
    ||u||_{L^{\nu}(\mathbb{R}^N)}\leq C||u||_{W^{s,p}(\mathbb{R}^N)}\quad\quad \text{ for all } u\in  W^{s,p}(\mathbb{R}^N).
\end{equation}
Let $E$ denote the completion of $C_0^{\infty}(\mathbb{R}^N)$ endowed with the norm
\begin{equation*}
    ||u||_E=\left([u]_{s,p}^p+||u||_{p,V}^p\right)^{\frac{1}{p}},\ \text{where} \quad ||u||_{p,V}=
    \Big(\int_{\mathbb{R}^N}V(x)|u(x)|^p dx\Big)^\frac{1}{p}.
\end{equation*}
\par Now we introduce the main space $W$ which is the completion  of $C_0^{\infty}(\mathbb{R}^N)$ with respect to the norm
\begin{equation*}
    ||u||_W=||u||_E+||u||_{L^q(\mathbb{R}^N,b)},\    \text{ where} \quad ||u||_{L^q(\mathbb{R}^N,b)}=\Big(\int_{\mathbb{R}^N}b(x)|u|^q dx\Big)^\frac{1}{q}.
\end{equation*}
It is proved in \cite[Lemma 2.2 ]{XZR1} that $W$ is a reflexive Banach space by using the fact that
$(E,\|\cdot\|_E)$ is a uniformly convex Banach space, which is proved in \cite[Lemma 10]{PXZ1} and  \cite[Lemma A.6]{PXZ}.
However, there are mistakes in the proof of both lemmas when $1<p<2$. For instance, since $p^{\prime}/p=1/(p-1)>1$ when $1<p<2$, applying $(a+b)^p\leq2^{p-1}(a^p+b^p)$
for all $a,b>0$ and $p\geq1$  we have
\begin{eqnarray}
&&\Big\|\frac{u+v}{2}\Big\|_E^{p^{\prime}}+\Big\|\frac{u-v}{2}\Big\|_E^{p^{\prime}}\nonumber\\
&=&\Big(\big[\frac{u+v}{2}\big]_{s,p}^p+\big\|\frac{u+v}{2}\big\|_{p,V}^p\Big)^{p^{\prime}/p}+
\Big(\big[\frac{u-v}{2}\big]_{s,p}^p+\big\|\frac{u-v}{2}\big\|_{p,V}^p\Big)^{p^{\prime}/p}\nonumber\\
&\leq& 2^{\frac{1}{p-1}-1}\Big(\big[\frac{u+v}{2}\big]_{s,p}^{p^{\prime}}+\big\|\frac{u+v}{2}\big\|_{p,V}^{p^{\prime}}
+\big[\frac{u-v}{2}\big]_{s,p}^{p^{\prime}}+\big\|\frac{u-v}{2}\big\|_{p,V}^{p^{\prime}}\Big)\nonumber\\
&=& 2^{\frac{1}{p-1}-1}\Big\{\Big\|\Big(\Big|\frac{u(x)+v(x)}{2}-\frac{u(y)+v(y)}{2}\Big|\big|x-y\big|
^{\frac{-N-ps}{p}}\Big)^{p^{\prime}}\Big\|_{L^{p-1}(\mathbb{R}^{2N})} \nonumber\\
&&\ \ \ \ \ \ \ +\Big\|\Big(V(x)^{\frac{1}{p}}\Big|\frac{u(x)+v(x)}{2}\Big|\Big)^{p^{\prime}}\Big\|_{L^{p-1}(\mathbb{R}^{N})}\nonumber\\
&&\ \ \ \ \ \ \ +\Big\|\Big(\Big|\frac{u(x)-v(x)}{2}-\frac{u(y)-v(y)}{2}\Big|\big|x-y\big|
^{\frac{-N-ps}{p}}\Big)^{p^{\prime}}\Big\|_{L^{p-1}(\mathbb{R}^{2N})} \nonumber\\
&&\ \ \ \ \ \ \ +\Big\|\Big(V(x)^{\frac{1}{p}}\Big|\frac{u(x)-v(x)}{2}\Big|\Big)^{p^{\prime}}\Big\|_{L^{p-1}(\mathbb{R}^{N})}\Big\},
\end{eqnarray}
which shows the last inequality in (5.3) of \cite[Lemma 10]{PXZ1} lacks a factor $2^{\frac{1}{p-1}-1}$. Besides, in (A.4) of \cite[Lemma A.6]{PXZ}
the equality that $2^{1/(1-p)}(\|u\|_E^p+\|v\|_E^p)^{1/(p-1)}=2^{1/(1-p)}$ should be $2^{1/(1-p)}(\|u\|_E^p+\|v\|_E^p)^{1/(p-1)}=1$,
so neither (5.3) of \cite[Lemma 10]{PXZ1} nor (A.4) of \cite[Lemma A.6]{PXZ} can imply the uniform convexity of $E$.
We shall give another proof by contradiction in the appendix.
\par From now on, $B_R$ denotes the ball of radius $R$ in $\mathbb{R}^N$ that is centered at the origin, $D^{s,p}(\mathbb{R}^N)$ is the completion of $C_0^\infty(\mathbb{R}^N)$ with respect to the norm $[u]_{s,p}$ and
$$W_M=\{u\in W:\ ||u||_E\leq M\}$$ for each fixed $M>0$.

\begin{lem}\label{EWL}{\rm(See\cite[Lemma 1]{PXZ1})}
The embeddings $E\hookrightarrow W^{s,p}(\mathbb{R}^N)\hookrightarrow L^{v}(\mathbb{R}^N)$ are continuous with
\begin{equation*}
  \min\{1,V_0\}||u||_{W^{s,p}(\mathbb{R}^N)}^p\leq ||u||_E^p
\end{equation*}
for all $u\in W$ and $v\in[p,p_s^*]$. Moreover, for any $R>0$ and $v\in[1,p_s^*]$, the embeddings $E\hookrightarrow\hookrightarrow L^v(B_R)$ is compact
for all $m\in[1,p_s^*)$.
\end{lem}
\par The following lemma is \cite[Lemma 2.1]{PXZ}, which is an application of \cite[Theorem 6.5]{NPV} and \cite[Corollary 7.2]{NPV}.
\begin{lem}\label{WEDL}
The embeddings $W\hookrightarrow E\hookrightarrow D^{s,p}(\mathbb{R}^N)\hookrightarrow L^{p_s^*}(\mathbb{R}^N)$ are continuous with
$[u]_{s,p}\leq||u||_E$ for all $u\in E$, $||u||_E\leq||u||_W$ for all $u\in W$,
and
\begin{equation*}
   ||u||_{L^{p_s^*}(\mathbb{R}^N)}\leq C_{p_s^*}[u]_{s,p},
\end{equation*}
for all $u\in D^{s,p}(\mathbb{R}^N)$. Moreover, for any $R>0$, the embeddings $E\hookrightarrow\hookrightarrow L^m(B_R)$ and $W\hookrightarrow\hookrightarrow L^m(B_R)$ are compact for all $m\in[1,p_s^*)$.
\end{lem}

\begin{lem}\label{WL} If the three conditions {\rm (\ref{C1})-(\ref{C3})} are satisfied, then
$W\hookrightarrow\hookrightarrow L^r(\mathbb{R}^N,a)$ is compact.
\end{lem}

\begin{proof}
Denote
\begin{equation*}
    s_1=\frac{\beta(p_s^*-r)}{p_s^*+\beta(p_s^*-p)+\gamma(p_s^*-q)},\ s_2=\frac{\gamma(p_s^*-r)}{p_s^*+\beta(p_s^*-p)+\gamma(p_s^*-q)}
\end{equation*}
and
\begin{equation*}
s_3=\frac{r+\beta(r-p)+\gamma(r-q)}{p_s^*+\beta(p_s^*-p)+\gamma(p_s^*-q)},\ s_4=\frac{p_s^*-r}{p_s^*+\beta(p_s^*-p)+\gamma(p_s^*-q)}.
\end{equation*}
One can check that $\sum\limits_{i=1}^4 s_i=1$.
\par If $p_s^*\geq q$, then by $\beta\geq0$, $0<\gamma\leq\frac{(r-p)\big(p_s^*+\beta(p_s^*-p)\big)}{(q-r)(p_s^*-p)}
<\frac{\beta(r-p)+r}{q-r}$ and condition (\ref{C1})
we have $s_1\geq0$ and $s_i>0$ for each $i=2,3,4$. So $\sum\limits_{i=1}^4 s_i=1$ implies $s_1\in[0,1)$
and $s_i\in(0,1)$ for each $i=2,3,4$.
\par If $p_s^*<q$, then $0<\gamma\leq\frac{(r-p)\big(p_s^*+\beta(p_s^*-p)\big)}{(q-r)(p_s^*-p)}
<\frac{\beta(r-p)+r}{q-r}<\frac{\beta(p_s^*-p)+p_s^*}{q-p_s^*}$ and
so $p_s^*+\beta(p_s^*-p)+\gamma(p_s^*-q)>0$. Combining this with $\beta\geq0$, $0<\gamma<\frac{\beta(r-p)+r}{q-r}$ and condition (\ref{C1})
yields $s_1\geq0$ and $s_i>0$ for each $i=2,3,4$, thus $\sum\limits_{i=1}^4 s_i=1$ implies $s_1\in[0,1)$
and $s_i\in(0,1)$ for each $i=2,3,4$.

\par Since $ps_1+qs_2+p_s^*s_3=r$, $\beta=s_1/s_4$ and $\gamma=s_2/s_4$,
applying H\"{o}lder's inequality with Lemma \ref{WEDL} and condition (\ref{C2}) we obtain

\begin{eqnarray}\label{aur}
||u||_{L^r(\mathbb{R}^N,a)}^r&=&\int_{\mathbb{R}^N} a(x)\n{u}^r \dx \nonumber \\ &=&\int_{\mathbb{R}^N}\n{u}^{r-ps_1-qs_2}\bre{V(x)\n{u}^p}^{s_1}\bre{b(x)\n{u}^q}^{s_2}
\bre{\frac{a(x)}{{{V(x)}^{s_1}}{b(x)}^{s_2}}}\dx \nonumber \\
&\leq& \Big(\int_{\mathbb{R}^N}\n{u}^{p_s^*}dx\Big)^{s_3}\Big(\int_{\mathbb{R}^N}{V(x)\n{u}^p}dx\Big)^{s_1}\Big(\int_{\mathbb{R}^N}{b(x)\n{u}^q}dx\Big)^{s_2}
\Big(\int_{\mathbb{R}^N}\frac{{a(x)}^{\frac{1}{s_4}}}{{{V(x)}^{\beta}}{b(x)}^{\gamma}}\dx\Big)^{s_4} \nonumber \\
&\leq&(C_{p_s^*}[u]_{s,p})^{p_s^*s_3}\Big(\int_{\mathbb{R}^N}{V(x)\n{u}^p}dx\Big)^{s_1}\Big(\int_{\mathbb{R}^N}{b(x)\n{u}^q}dx\Big)^{s_2}
\Big(\int_{\mathbb{R}^N}\frac{{a(x)}^{\frac{1}{s_4}}}{{{V(x)}^{\beta}}{b(x)}^{\gamma}}\dx\Big)^{s_4} \nonumber \\
&\leq& C_1([u]_{s,p})^{p_s^*s_3}\Big(\int_{\mathbb{R}^N}{V(x)\n{u}^p}dx\Big)^{s_1}\Big(\int_{\mathbb{R}^N}{b(x)\n{u}^q}dx\Big)^{s_2}
\Big(\int_{\mathbb{R}^N}\frac{{a(x)}^{\frac{1}{s_4}}}{{b(x)}^{\gamma}}\dx\Big)^{s_4}
\end{eqnarray} with constant $C_1=C_{p_s^*}^{p_s^*s_3}V_0^{-s_1}>0$.
\par By $ps_1+qs_2+p_s^*s_3=r$ and (\ref{aur}) we have $||u||_{L^r(\mathbb{R}^N,a)}\leq C_1^\frac{1}{r}\Big(\int_{\mathbb{R}^N}\frac{{a(x)}^{\frac{1}{s_4}}}{{b(x)}^{\gamma}}\dx\Big)^{\frac{s_4}{r}}||u||_W$, then
the embedding $W\hookrightarrow L^r(\mathbb{R}^N,a)$ is continuous. Next we will show that $W\hookrightarrow\hookrightarrow L^r(\mathbb{R}^N,a)$
is compact.
\par Indeed, by condition (\ref{C3}) for any $\varepsilon>0$, there exists an $R_1>0$ such that
$\int_{\mathbb{R}^N\setminus B_R}\frac{{a(x)}^{\frac{1}{s_4}}}{{b(x)}^{\gamma}}\dx<\varepsilon^{1/s_4}$
for all $R\geq R_1$. Fix $R_1>0$ and let $\{u_n\}_n$ be a bounded sequence in $W$, then by Lemma \ref{EWL}, \cite[Theorem 6.5]{NPV}
and \cite[Corollary 7.2]{NPV} as in \cite[Theorem 2.1]{XZR1} we obtain a subsequence of $\{u_n\}_n$, which is also denoted by $\{u_n\}_n$ for convenience,
satisfying $u_n\rightharpoonup u$ weakly in $W\cap L^{p_s^{*}}$ and $a(x)|u_n-u|^r\rightarrow0$ a.e. in $B_{R_1}$ as $n\rightarrow\infty$.
Since $u_n\rightharpoonup u$ weakly in $W\cap L^{p_s^{*}}$, applying condition (\ref{C3}) and the argument of (\ref{aur}) yields
\begin{equation}\label{RNBR1}\int_{\mathbb{R}^N\setminus B_{R_1}} a(x)|u_n(x)-u(x)|^r dx\leq C_1\Big(\int_{\mathbb{R}^N\setminus B_{R_1}}\frac{{a(x)}
^{\frac{1}{s_4}}}{{b(x)}^{\gamma}}\dx\Big)^{s_4}||u_n-u||_W^r\leq C_2 \varepsilon\end{equation}
and
\begin{equation}\label{Uaunur}\int_U a(x)|u_n(x)-u(x)|^r dx\leq C_1\Big(\int_{\mathbb{R}^N}\frac{{a(x)}
^{\frac{1}{s_4}}}{{b(x)}^{\gamma}}\dx\Big)^{s_4}||u_n-u||_W^r\leq C_3<\infty\end{equation}
for each measurable subset $U\subset B_{R_1}$ with constants $C_2, C_3>0$. Then (\ref{Uaunur}) and
the Vitali convergence theorem imply
$$\lim\limits_{n\rightarrow\infty}\int_{B_{R_1}} a(x)|u_n(x)-u(x)|^r dx=0.$$
Thus, for the above $\varepsilon>0$, there exists an integer $N_1>0$ such that
\begin{equation}\label{BR1aunur}\int_{B_{R_1}} a(x)|u_n(x)-u(x)|^r dx<\varepsilon\end{equation}
for all $n\geq N_1$. Hence, using (\ref{RNBR1}) and (\ref{BR1aunur}) obtains
$$\int_{\mathbb{R}^N}a(x)|u_n-u|^r\dx=\int_{\mathbb{R}^N\setminus B_{R_1}} a(x)|u_n(x)-u(x)|^r dx + \int_{B_{R_1}} a(x)|u_n(x)-u(x)|^r dx
<(C_2+1)\varepsilon$$ for all $n\geq N_1$,
which proves $W\hookrightarrow\hookrightarrow L^r(\mathbb{R}^N,a)$ is compact.
\end{proof}

\section{Proof of Theorem \ref{Th1}}
\par As in \cite{XZR1} for each $u\in W$ we define $I(u)=J(u)-H(u)$, whose critical points are weak solutions of problem (\ref{problem1.1}), where
\begin{eqnarray}&&J(u)=\frac{1}{p}\int\int_{\mathbb{R}^{2N}}\frac{|u(x)-u(y)|^p}{|x-y|^{N+ps}}dxdy+\frac{1}{p}\int_{\mathbb{R}^{N}}V(x)|u(x)|^pdx
+\frac{1}{q}\int_{\mathbb{R}^{N}}b(x)|u(x)|^qdx,\nonumber\\
&&H(u)=\frac{\lambda}{r}\int_{\mathbb{R}^{N}}a(x)|u(x)|^rdx.\nonumber\end{eqnarray}
\begin{lem}\label{coercive}
Under the conditions {\rm(\ref{C1})-(\ref{C3})}, the function $I$ is coercive and weakly lower semi-continuous in $W_M$
for each constant $M>0$.
\end{lem}
\begin{proof} Let $M>0$ be a constant. By (\ref{aur}) for each $u\in W_M$ we have
\begin{eqnarray}\label{Hu}
H(u)&=&\frac{\lambda}{r}\int_{\mathbb{R}^N} a(x)\n{u}^r\dx\nonumber \\
&\leq&\frac{\lambda C_1}{r}\Big([u]_{s,p}^p\Big)^{\frac{p_s^*s_3}{p}}\Big(\int_{\mathbb{R}^N}{V(x)\n{u}^p}dx\Big)^{s_1}
\Big(\int_{\mathbb{R}^N}{b(x)\n{u}^q}dx\Big)^{s_2}\Big(\int_{\mathbb{R}^N}\frac{{a(x)}^{\frac{1}{s_4}}}{{b(x)}^{\gamma}}\dx\Big)^{s_4}\nonumber\\
&\leq& C_4 ||u||_E^{p_s^*s_3+ps_1}||u||_{L^q(\mathbb{R}^N,b)}^{qs_2}
\end{eqnarray}
with constant $C_4=\frac{\lambda C_1}{r}\Big(\int_{\mathbb{R}^N}\frac{{a(x)}^{\frac{1}{s_4}}}{{b(x)}^{\gamma}}\dx\Big)^{s_4}>0$.
Then by $||u||_E\leq M$ we get
\begin{eqnarray}\label{2.4}
I(u)&\geq&\frac{1}{p}||u||_E^p+\frac{1}{q}||u||_{L^q(\mathbb{R}^N,b)}^q-C_4 ||u||_E^{p_s^*s_3+ps_1}||u||_{L^q(\mathbb{R}^N,b)}^{qs_2}
\nonumber\\ &\geq&\frac{1}{q}||u||_{L^q(\mathbb{R}^N,b)}^q-C_4 M^{p_s^*s_3+ps_1}||u||_{L^q(\mathbb{R}^N,b)}^{qs_2}
\nonumber\\ &\geq&\frac{1}{q}(||u||_W-M)^q-C_4 M^{p_s^*s_3+ps_1}||u||_W^{qs_2}.
\end{eqnarray}
This, together with $q>qs_2$, implies $I$ is coercive in $W_M$ for each constant $M>0$.

\par Note that $W_M\subset W$, then \cite[Lemma 3.2]{XZR1} implies the functional $J$ is weakly lower semi-continuous in $W_M$
with $J\in C^1(W_M,\mathbb{R})$.
Also, Lemma \ref{WL} implies $W_M\hookrightarrow\hookrightarrow L^r(\mathbb{R}^N,a)$ is compact, then by applying the similar
argument as in \cite[Lemma 3.3]{XZR1} one can get the functional $H$ is weakly continuous in $W_M$ with $H\in C^1(W_M,\mathbb{R})$. Hence, the functional $I$
is weakly lower semi-continuous in $W_M$ with $I\in C^1(W_M,\mathbb{R})$.
\end{proof}

\par For each fixed $M>0$ we define $$\widetilde{W}_M:=\{u\in W_M:\ ||u||_{L^r(\mathbb{R}^N,a)}=1\}\ \ {\rm and}\ \
\lambda_M^*:=\inf_{u\in \widetilde{W}_M}rJ(u).$$
Applying Lemma \ref{WL} with a similar argument as in \cite[Lemma 3.4]{XZR1} we have
\begin{lem}\label{minimizer} Let $M>0$ be fixed. Then $\inf_{u\in \widetilde{W}_M}J(u)$ can be achieved at some $u_M\in \widetilde{W}_M$ and
$\lambda_M^*=r\cdot\inf_{u\in \widetilde{W}_M}J(u)=rJ(u_M)>0$. Moreover, $|u_M|$ is also a minimizer of $\inf_{u\in\widetilde{W}_M}J(u)$,
which means that $J(|u_M|)=J(u_M)$.
\end{lem}
\par The following theorem is the modification of the mountain pass theorem of Ambrosetti-Rabinowitz
(see \cite[Theorem A.3]{AP1}), which will be used to prove Theorem \ref{Th1} (ii).
\begin{thm}\label{mountainpass}
Let $(X,||\cdot||_x)$ and  $(Y,||\cdot||_Y)$ be two Banach spaces. Suppose $X$ can be continuously embedded into $Y$. Let $\Phi: X\rightarrow\mathbb{R}$ be a $C^1$ functional with $\Phi(0)=0$. Assume that there exist $\rho,\alpha>0$ and $e\in X$ such that $||e||_Y>\rho$, $\Phi(e)<\alpha$ and $\Phi(u)\geq\alpha$ for all $u\in X$ with $||u||_Y=\rho$. Then there exists a sequence $\{u_n\}\subset X$ such that for all $n$
\begin{equation*}
    c\leq\Phi(u_n)\leq c+\frac{1}{n} \quad \text{and}\quad ||\Phi'(u_n)||_{X'}\leq\frac{2}{n},
\end{equation*}
where
\begin{equation*}
    c=\inf_{\gamma\in\Gamma}\max_{t\in[0,1]}\Phi(\gamma(t))\quad\text{and}\quad \Gamma=\{\gamma\in C([0,1];X):\gamma(0)=0,\gamma(1)=e\}.
\end{equation*}
\end{thm}
Now for each $e\in W\backslash\{0\}$ we find a lower bound of the functional $I$ on the boundary of $W_M$ with $0<M<||e||_E$
to establish the next lemma, which is similar to \cite[Lemma 3.5]{XZR1}, but its proof is different.
\begin{lem}\label{boundedbelow} Suppose the conditions {\rm(\ref{C1})-(\ref{C3})} are satisfied. Then for each $e\in W\setminus\{0\}$,
there exist $\rho\in(0,||e||_E)$ and $\alpha>0$ such that
\begin{equation*}
I(u)\geq\alpha>0,
\end{equation*}
for all $u\in W$ with $||u||_E=\rho$.
\end{lem}
\begin{proof} We divide the proof of this lemma into two cases as follows.
\par Case 1\quad If $\beta\in[0,\infty)$ and $\gamma\in\Big(0,\frac{(r-p)\big(p_s^*+\beta(p_s^*-p)\big)}{(q-r)(p_s^*-p)}\Big)$
when $a, b\in L^1_{loc}(\mathbb{R}^N)$, by (\ref{aur}) and (\ref{Hu}) for each $u\in W$ we have
\begin{eqnarray}\label{Iu1}
I(u)\geq\frac{1}{p}||u||_E^p+\frac{1}{q}||u||_{L^q(\mathbb{R}^N,b)}^q-C_4\Big([u]_{s,p}^p\Big)^{\frac{p_s^*s_3}{p}}
\Big(||u||_{p,V}^p\Big)^{s_1}||u||_{L^q(\mathbb{R}^N,b)}^{qs_2}.
\end{eqnarray}
Denote $\rho=||u||_E$ and $x=[u]_{s,p}^p$, then $||u||_{p,V}^p=\rho^p-x$ and direct computation yields
\begin{eqnarray}\label{uspupv}\Big([u]_{s,p}^p\Big)^{\frac{p_s^*s_3}{p}}\Big(||u||_{p,V}^p\Big)^{s_1}=x^{\frac{p_s^*s_3}{p}}(\rho^p-x)^{s_1}\leq
x_0^{\frac{p_s^*s_3}{p}}(\rho^p-x_0)^{s_1}=C_5(\rho^p)^{\frac{p_s^*s_3}{p}+s_1}
\end{eqnarray}
with $x_0=\frac{p_s^*s_3\rho^p}{p_s^*s_3+ps_1}$ and $C_5=\Big(\frac{p_s^*s_3}{p_s^*s_3+ps_1}\Big)^{\frac{p_s^*s_3}{p}}
\Big(\frac{ps_1}{p_s^*s_3+ps_1}\Big)^{s_1}$.
\par Applying (\ref{Iu1}) and (\ref{uspupv}) with $y=||u||_{L^q(\mathbb{R}^N,b)}^q$ and $C_6=C_4C_5$ we get
\begin{eqnarray}\label{Iu2}
I(u)&\geq&\frac{1}{p}\rho^p+\frac{1}{q}y-C_6(\rho^p)^{\frac{p_s^*s_3}{p}+s_1} y^{s_2}
\nonumber\\ &\geq&\frac{1}{p}\rho^p+\frac{1}{q}y_0-C_6(\rho^p)^{\frac{p_s^*s_3}{p}+s_1} y_0^{s_2}
\nonumber\\ &=&\frac{1}{p}\rho^p-C_7(\rho^p)^{\frac{1}{1-s_2}\big(\frac{p_s^*s_3}{p}+s_1\big)},
\end{eqnarray}
where $y_0=\Big(C_6qs_2(\rho^p)^{\frac{p_s^*s_3}{p}+s_1}\Big)^{\frac{1}{1-s_2}}$ and
$C_7=(1-s_2)C_6^{\frac{1}{1-s_2}}(qs_2)^{\frac{s_2}{1-s_2}}$.
\par Note that $\gamma\in\Big(0,\frac{(r-p)\big(p_s^*+\beta(p_s^*-p)\big)}{(q-r)(p_s^*-p)}\Big)$ implies
$\frac{1}{1-s_2}\big(\frac{p_s^*s_3}{p}+s_1\big)>1$, then by (\ref{Iu2}) for each $e\in W\setminus\{0\}$
there exists $0<\rho<\min\Big\{||e||_E, \Big(pC_7\Big)^{-1\big/\big[\frac{p}{1-s_2}\big(\frac{p_s^*s_3}{p}+s_1\big)-p\big]}\Big\}$
such that $$I(u)\geq\rho^p\Big(\frac{1}{p}-C_7\rho^{\frac{p}{1-s_2}\big(\frac{p_s^*s_3}{p}+s_1\big)-p}\Big)=:\alpha>0.$$
\par Case 2\quad If  $\beta\in[0,\infty)$ and $\gamma=\frac{(r-p)\big(p_s^*+\beta(p_s^*-p)\big)}{(q-r)(p_s^*-p)}$
when $a\in L^k(\mathbb{R}^N)$ with $k>1$ and $b\in L^1_{loc}(\mathbb{R}^N)$, we can take sufficiently small
$t_5\in(0,1/k)$, $t_4=s_4(1-kt_5)$, $t_3=\frac{1}{p_s^*-p}[(p+p\gamma-q\gamma)t_4+pt_5+(r-p)]$,
$t_2=\gamma t_4$ and $t_1=1-t_3-(1+\gamma)t_4-t_5$ satisfying $pt_1+qt_2+p_s^*t_3=r$
and $\sum\limits_{i=1}^5 t_i=1$ with $t_i>0$, and use the same argument as in (\ref{aur}) to obtain
$$||u||_{L^r(\mathbb{R}^N,a)}^r\leq C_8([u]_{s,p})^{p_s^*t_3}\Big(\int_{\mathbb{R}^N}{V(x)\n{u}^p}dx\Big)^{t_1}\Big(\int_{\mathbb{R}^N}{b(x)\n{u}^q}dx\Big)^{t_2}
\Big(\int_{\mathbb{R}^N}\frac{{a(x)}^{\frac{1}{s_4}}}{{b(x)}^{\gamma}}\dx\Big)^{t_4}$$ with constant $C_8=C_{p_s^*}^{p_s^*t_3}V_0^{-t_1}||a||_{L^k(\mathbb{R}^N)}^{t_5}>0$. In this case it can be checked that
$$\frac{1}{1-t_2}\big(\frac{p_s^*t_3}{p}+t_1\big)>\frac{1}{1-s_2}\big(\frac{p_s^*s_3}{p}+s_1\big)=1,$$
then applying the argument of Case 1 yields that
$$I(u)\geq\rho^p\Big(\frac{1}{p}-C_9\rho^{\frac{p}{1-t_2}\big(\frac{p_s^*t_3}{p}+t_1\big)-p}\Big)=:\alpha>0$$
for all $u\in W$ with $||u||_E=\rho\in\Big(0, \min\Big\{||e||_E, \Big(pC_9\Big)^{-1\big/\big[\frac{p}{1-t_2}\big(\frac{p_s^*t_3}{p}+t_1\big)-p\big]}\Big\}\Big)$, where $C_9>0$ is a constant.
\end{proof}

\par Now we first give the proof of Theorem \ref{Th1}.
\par {\bf Proof of Theorem \ref{Th1}}. (i) Suppose $u\neq0$ is a nontrivial weak solution
of problem (\ref{problem1.1}). Since condition (\ref{C1}) implies $\frac{\gamma(r-qs_2)}{\gamma(1-s_2)-s_2}\in[p,p_s^*]$
and condition (\ref{C3}) implies $\frac{r-qs_2}{1-s_2}\in(0,p]$,
taking $\varphi=u$ in Definition \ref{solution} and applying H\"{o}lder's inequality with inequality (\ref{uLv}),
Lemma \ref{EWL} and \cite[inequality(3.1)]{XZR1} we have
\begin{eqnarray} \label{min1V0}
||u||_E^p&=&\lambda\int_{\mathbb{R}^N}a(x)|u(x)|^rdx-\int_{\mathbb{R}^N}b(x)|u(x)|^qdx.
\nonumber \\ &=&\int_{\mathbb{R}^N}\Big(\lambda a(x)|u|^{\frac{(q-r)s_2}{1-s_2}}-b(x)|u|^{\frac{q-r}{1-s_2}}\Big)|u|^{\frac{r-qs_2}{1-s_2}}dx
\nonumber \\ &\leq&\int_{\mathbb{R}^N}\lambda^{\frac{1}{1-s_2}}\Big(\frac{a(x)^{\frac{1}{1-s_2}}}{b(x)^{\frac{s_2}{1-s_2}}}\Big)|u|^{\frac{r-qs_2}{1-s_2}}dx
\nonumber \\ &\leq&\lambda^{\frac{1}{1-s_2}}\Big(\int_{\mathbb{R}^N}\frac{a(x)^\frac{1}{s_4}}{b(x)^{\gamma}}dx\Big)^\frac{s_2}{\gamma(1-s_2)}
||u||_{L^{\frac{\gamma(r-qs_2)}{\gamma(1-s_2)-s_2}}(\mathbb{R}^N)}^{\frac{r-qs_2}{1-s_2}}
\nonumber \\ &\leq&\lambda^{\frac{1}{1-s_2}}\Big(\int_{\mathbb{R}^N}\frac{a(x)^\frac{1}{s_4}}{b(x)^{\gamma}}dx\Big)^\frac{s_2}{\gamma(1-s_2)}
\Big(\big(\min\{1,V_0\}\big)^{-\frac{1}{p}}C||u||_E\Big)^{\frac{r-qs_2}{1-s_2}}
\end{eqnarray} with the same constants $s_2$ and $s_4$ as in (\ref{aur}).
\par Case 1\quad If  $u$ is a nontrivial weak solution in $W$ when $a\in L^k(\mathbb{R}^N)$ with $k>1$ and $b\in L^1_{loc}(\mathbb{R}^N)$, then
$\gamma=\frac{(r-p)\big(p_s^*+\beta(p_s^*-p)\big)}{(q-r)(p_s^*-p)}$ implies $\frac{r-qs_2}{1-s_2}=p$. By (\ref{min1V0}) we have
$$\lambda\geq\Big(\int_{\mathbb{R}^N}\frac{a(x)^\frac{1}{s_4}}{b(x)^{\gamma}}dx\Big)^{-\frac{s_2}{\gamma}}
\Big(\big(\min\{1,V_0\}\big)^{-\frac{1}{p}}C\Big)^{-p(1-s_2)}=:\lambda_*.$$
So problem \eqref{problem1.1} has only the trivial weak solution in $W$ if $\lambda<\lambda_*$ when $a\in L^k(\mathbb{R}^N)$ with $k>1$ and $b\in L^1_{loc}(\mathbb{R}^N)$.
\par Case 2\quad Let $M>0$ be fixed. If  $u$ is a nontrivial weak solution in $W_M$ when $a, b\in L^1_{loc}(\mathbb{R}^N)$, then
$\gamma\in\Big(0,\frac{(r-p)\big(p_s^*+\beta(p_s^*-p)\big)}{(q-r)(p_s^*-p)}\Big)$ implies $\frac{r-qs_2}{1-s_2}>p$.
By (\ref{min1V0}) and $||u||_E\leq M$ we have
$$\lambda\geq\Big(\int_{\mathbb{R}^N}\frac{a(x)^\frac{1}{s_4}}{b(x)^{\gamma}}dx\Big)^{-\frac{s_2}{\gamma}}
\Big(\big(\min\{1,V_0\}\big)^{-\frac{1}{p}}C\Big)^{qs_2-r} M^{(p-q)s_2+p-r}=:\lambda_M.$$
So problem \eqref{problem1.1} has only the trivial weak solution in $W_M$ if $\lambda<\lambda_M$ when $a, b\in L^1_{loc}(\mathbb{R}^N)$.
\par (ii) Let $M>0$ be an arbitrary constant and $\lambda_M^*$ be the same as in Lemma \ref{minimizer}.
 Recall that $W$ is a reflexive Banach space, it can be checked that $W_M$ is a closed and convex subset of $W$,
 then $W_M$ is a weakly closed subset of $W$ and also a reflexive Banach subspace of $W$. Since Lemma \ref{coercive} and \ref{boundedbelow}
 imply the functional $I$ is weakly lower semi-continuous, bounded below and coercive in $W_M$ for all $\lambda>0$, by \cite[Theorem 1.2]{St}
 there exists a $u_M^*\in W_M$ such that $I(u_M^*)=\inf\limits_{u\in W_M}I(u)$. When $\lambda>\lambda_M^*$, applying Lemma \ref{minimizer} and
 similar argument of \cite[Theorem 3.1]{XZR1} yields $I(u_M^*)<0$.
 \par Taking $e=u_M^*$ in Lemma \ref{boundedbelow} we can check that the functional $I$ satisfies the assumptions of Theorem \ref{mountainpass}.
Then for all $\lambda>\lambda_M^*$  by Theorem \ref{mountainpass} there is a sequence $\{u_{n,M}\}_{n\geq1}\subset W_M$ such that
\begin{equation*}
I(u_{n,M})\rightarrow c_M\quad \text{and}\quad ||I'(u_{n,M})||_{W_M^{'}}\rightarrow 0 \quad\text{as}\quad n\rightarrow\infty,
\end{equation*}
where
\begin{equation*}
    c_M=\inf_{\gamma\in\Gamma}\max_{t\in[0,1]}I(\gamma(t))\quad\text{and}\quad \Gamma=\{\gamma\in C([0,1]; W_M):\gamma(0)=0,\gamma(1)=u_M^*\}.
\end{equation*}
So coerciveness of the functional $I$ in $W_M$ and reflexivity of $W_M$ imply $\{u_{n,M}\}_{n\geq1}$ has a subsequence,
which is also denoted by $\{u_{n,M}\}$ for convenience, such that $u_{n,M}\rightharpoonup u_M$ weakly in $W_M$ and thus
$u_{n,M}\rightarrow u_M$ strongly in $L^r(\mathbb{R}^N,a)$ by Lemma \ref{WL}. By the similar proof of \cite[Theorem 3.3]{XZR1} we obtain
$I(u_M)=\lim\limits_{n\rightarrow\infty}I(u_{n,M})=c_M>0$.
\par Moreover, using the similar argument of \cite[Corollary 3.4]{XZR1} with Theorem \ref{mountainpass} and Lemma \ref{boundedbelow}
we can show that problem \eqref{problem1.1}
has at least two nontrivial nonnegative weak solutions in $W_M$ in which one has negative energy and another
has positive energy if $\lambda>\lambda_M^*$. So Theorem \ref{Th1} is proved.
\section{Proof of Theorem \ref{continuity}}
\begin{proof} Suppose that the three conditions {\rm (\ref{C1})-(\ref{C3})} are satisfied and that
$a(x), b(x), V(x)\in L^{\infty}(\Omega)$ for any bounded domain $\Omega\subset\mathbb{R}^N$. Let
$u\in W$ be a weak solution of problem \eqref{problem1.1} and $\Omega\subset\mathbb{R}^N$ be
a bounded domain. Define $\widetilde{u}(x):=u(x)$ for each $x\in\Omega$ and $\widetilde{u}(x):=0$ for each
$x\in\mathbb{R}^N\setminus\Omega$. Then $(-\Delta)_p^s\widetilde{u}=f(x,\widetilde{u})$ for each $x\in\Omega$ with
$f(x,\widetilde{u})=\lambda a(x)|\widetilde{u}|^{r-2}\widetilde{u}-b(x)|\widetilde{u}|^{q-2}\widetilde{u}-V(x)|\widetilde{u}|^{p-2}\widetilde{u}$
and $\|\widetilde{u}\|_W\leq\|u\|_W<\infty$.
\par Given $k>0$, $\mu\geq1$ and $\nu>0$, define $g_{\mu,\nu}(t):=t^{\mu}(t_k)^{\nu}$ with $t_k=\min\{t,k\}$ for all $t\geq0$. We claim that
\begin{equation} \label{Gmunu} G_{\mu,\nu}(t):=\int_0^t\big(g_{\mu,\nu}^{\prime}(\tau)\big)^{\frac{1}{p}}d\tau\geq\frac{p(\mu+\nu)^{\frac{1}{p}}}{\mu+\nu+p-1}
g_{\frac{\mu+p-1}{p},\frac{\nu}{p}}(t).\end{equation}
Indeed, if $t\leq k$, then $g_{\mu,\nu}(t)=t^{\mu+\nu}$ and so
$$G_{\mu,\nu}(t)=\int_0^t\big((\mu+\nu)\tau^{\mu+\nu-1}\big)^{\frac{1}{p}}d\tau
=\frac{p(\mu+\nu)^{\frac{1}{p}}}{\mu+\nu+p-1}t^{\frac{\mu+\nu+p-1}{p}}
=\frac{p(\mu+\nu)^{\frac{1}{p}}}{\mu+\nu+p-1}
g_{\frac{\mu+p-1}{p},\frac{\nu}{p}}(t).$$
If $t>k$, then
\begin{eqnarray}
G_{\mu,\nu}(t)&=&\int_0^k\big((\mu+\nu)\tau^{\mu+\nu-1}\big)^{\frac{1}{p}}d\tau+
\int_k^tk^{\frac{\nu}{p}}(\mu t^{\mu-1})^{\frac{1}{p}}d\tau \nonumber \\
&=&\frac{p(\mu+\nu)^{\frac{1}{p}}}{\mu+\nu+p-1}k^{\frac{\mu+\nu+p-1}{p}}+k^{\frac{\nu}{p}}
\frac{p\mu^{\frac{1}{p}}}{\mu+p-1}\big(t^{\frac{\mu+p-1}{p}}-k^{\frac{\mu+p-1}{p}}\big).
\nonumber \end{eqnarray}
Since $\mu\geq1, \nu>0$ and $p\geq1$ one can check that $(\mu+\nu)(\mu+p-1)^p\leq\mu(\mu+\nu+p-1)^p$,
which implies  $\frac{\mu+p-1}{\mu+\nu+p-1}\leq\Big(\frac{\mu}{\mu+\nu}\Big)^{\frac{1}{p}}$. Therefore,
$$G_{\mu,\nu}(t)\geq\frac{p(\mu+\nu)^{\frac{1}{p}}}{\mu+\nu+p-1}k^{\frac{\nu}{p}}t^{\frac{\mu+p-1}{p}}
=\frac{p(\mu+\nu)^{\frac{1}{p}}}{\mu+\nu+p-1}g_{\frac{\mu+p-1}{p},\frac{\nu}{p}}(t)$$ if $t>k$.
So the claim (\ref{Gmunu}) follows.
\par Now by \cite[Lamma A.2]{BP} we have
\begin{eqnarray}\label{Gmunupsp}&&[G_{\mu,\nu}(\widetilde{u})]_{s,p}^p\nonumber\\
&=&\int_{\mathbb{R}^N}\int_{\mathbb{R}^N}\frac{|G_{\mu,\nu}(\widetilde{u})(x)-G_{\mu,\nu}(\widetilde{u})(y)|^p}{|x-y|^{N+ps}}dxdy\nonumber\\
&\leq&\int_{\mathbb{R}^N}\int_{\mathbb{R}^N}
\frac{|\widetilde{u}(x)-\widetilde{u}(y)|^{p-2}|\big(\widetilde{u}(x)-\widetilde{u}(y)\big)
\big(g_{\mu,\nu}(\widetilde{u})(x)-g_{\mu,\nu}(\widetilde{u})(y)\big)}{|x-y|^{N+ps}}dxdy\nonumber\\
&=&\langle(-\Delta)_p^s(\widetilde{u}),g_{\mu,\nu}(\widetilde{u})\rangle\nonumber\\
&=&\langle f(x,\widetilde{u}),g_{\mu,\nu}(\widetilde{u})\rangle\nonumber\\
&\leq&\int_{\Omega}\lambda a(x)|\widetilde{u}|^{r+\mu-1}|\widetilde{u}_k|^{\nu}dx+
\int_{\Omega}V(x)|\widetilde{u}|^{p+\mu-1}|\widetilde{u}_k|^{\nu}dx
+ \int_{\Omega}b(x)|\widetilde{u}|^{q+\mu-1}|\widetilde{u}_k|^{\nu}dx.
\end{eqnarray}
Applying \cite[Theorem 1]{MS} or \cite[Theorem 6.5]{NPV} with (\ref{Gmunu}) and (\ref{Gmunupsp}) obtains
\begin{eqnarray}
&&\Big(\int_{\Omega}\Big(\frac{p(\mu+\nu)^{\frac{1}{p}}}{\mu+\nu+p-1}\Big)^{p_s^*}|g_{\frac{\mu+p-1}{p},\frac{\nu}{p}}
(\widetilde{u})|^{p_s^*}dx\Big)^{\frac{p}{p_s^*}}\nonumber \\
&\leq&\Big(\int_{\Omega}|G_{\mu,\nu}(\widetilde{u})|^{p_s^*}dx\Big)^{\frac{p}{p_s^*}}\nonumber \\
&\leq& \widetilde{C}[G_{\mu,\nu}(\widetilde{u})]_{s,p}^p \nonumber \\
&\leq&\widetilde{C}\Big(\int_{\Omega}\lambda a(x)|\widetilde{u}|^{r+\mu-1}|\widetilde{u}_k|^{\nu}dx+
\int_{\Omega}V(x)|\widetilde{u}|^{p+\mu-1}|\widetilde{u}_k|^{\nu}dx
+ \int_{\Omega}b(x)|\widetilde{u}|^{q+\mu-1}|\widetilde{u}_k|^{\nu}dx\Big),
\end{eqnarray}
where $\widetilde{C}$ depends only on $N, p$ and $s$, which implies
\begin{eqnarray}\label{uukC0}
&&\Big(\int_{\Omega}|\widetilde{u}|^{\frac{(\mu+p-1)p_s^*}{p}}|\widetilde{u}_k|^{\frac{\nu p_s^*}{p}}dx\Big)^{\frac{p}{p_s^*}}\nonumber \\
&\leq&C_0\Big(\int_{\Omega}\lambda a(x)|\widetilde{u}|^{r+\mu-1}|\widetilde{u}_k|^{\nu}dx+
\int_{\Omega}V(x)|\widetilde{u}|^{p+\mu-1}|\widetilde{u}_k|^{\nu}dx+\int_{\Omega}b(x)|\widetilde{u}|^{q+\mu-1}|\widetilde{u}_k|^{\nu}dx\Big)
\end{eqnarray} with $C_0=\frac{(\mu+\nu+p-1)^p}{p^p(\mu+\nu)}\widetilde{C}>0$.
\par Since $[\widetilde{u}]_{s,p}\leq\|\widetilde{u}\|_W<\infty$, again by
\cite[Theorem 1]{MS} or \cite[Theorem 6.5]{NPV} we get $\widetilde{u}\in L^{p_s^*}(\Omega)$, and then by
$1<p<r<\min\{q,p_s^*\}$ we have
$\widetilde{u}\in L^{q+\mu-1}(\Omega)\subset L^{r+\mu-1}(\Omega)\subset L^{p+\mu-1}(\Omega)$ if $1\leq\mu\leq1+p_s^*-\min\{q,p_s^*\}$.
Thus, for each $1\leq\mu\leq1+p_s^*-\min\{q,p_s^*\}$ there must exist a positive number $K_0\geq1$ such that
\begin{eqnarray}\label{geqK0}\int_{\Omega\cap\{|\widetilde{u}|\geq K_0\}}|\widetilde{u}|^{p+\mu-1}dx
&\leq&\int_{\Omega\cap\{|\widetilde{u}|\geq K_0\}}|\widetilde{u}|^{r+\mu-1}dx \nonumber \\
&\leq&\int_{\Omega\cap\{|\widetilde{u}|\geq K_0\}}|\widetilde{u}|^{q+\mu-1}dx  \nonumber \\
&\leq&\Big(4C_0\big(\lambda\|a\|_{L^{\infty}(\Omega)}+\|V\|_{L^{\infty}(\Omega)}+
\|b\|_{L^{\infty}(\Omega)}\big)\Big)^{\frac{p_s^*}{p-p_s^*}},\end{eqnarray}
which, together with H\"{o}lder's inequality and (\ref{aur}), implies
\begin{eqnarray} \label{auuk}
&&\int_{\Omega}\lambda a(x)|\widetilde{u}|^{r+\mu-1}|\widetilde{u}_k|^{\nu}dx\nonumber \\
&=&\int_{\Omega\cap\{|\widetilde{u}|\geq K_0\}}\lambda a(x)|\widetilde{u}|^{r+\mu-1}|\widetilde{u}_k|^{\nu}dx
+\int_{\Omega\cap\{|\widetilde{u}|<K_0\}}\lambda a(x)|\widetilde{u}|^{r+\mu-1}|\widetilde{u}_k|^{\nu}dx\nonumber \\
&\leq&\lambda\|a\|_{L^{\infty}(\Omega)}\Big
(\int_{\Omega\cap\{|\widetilde{u}|\geq K_0\}}|\widetilde{u}|^{r+\mu-1}dx\Big)^{1-\frac{p}{p_s^*}}\Big
(\int_{\Omega\cap\{|\widetilde{u}|\geq K_0\}}|\widetilde{u}|^{r+\mu-1}|\widetilde{u}_k|
^{\frac{\nu p_s^*}{p}}dx\Big)^{\frac{p}{p_s^*}}\nonumber \\
&& + K_0^{\mu+\nu-1}\lambda\int_{\Omega}a(x)|\widetilde{u}|^rdx \nonumber\\
&\leq&\frac{1}{4C_0}\Big
(\int_{\Omega\cap\{|\widetilde{u}|\geq K_0\}}|\widetilde{u}|^{r+\mu-1}|\widetilde{u}_k|
^{\frac{\nu p_s^*}{p}}dx\Big)^{\frac{p}{p_s^*}}+K_0^{\mu+\nu-1}\lambda C_1\Big(\int_{\mathbb{R}^N}\frac{{a(x)}^{\frac{1}{s_4}}}{{b(x)}^{\gamma}}\dx\Big)^{s_4}||u||_W^r
\end{eqnarray}
with
\begin{eqnarray}\label{Vuuk}
&&\int_{\Omega}V(x)|\widetilde{u}|^{p+\mu-1}|\widetilde{u}_k|^{\nu}dx \nonumber \\
&=&\int_{\Omega\cap\{|\widetilde{u}|\geq K_0\}}V(x)|\widetilde{u}|^{p+\mu-1}|\widetilde{u}_k|^{\nu}dx
+\int_{\Omega\cap\{|\widetilde{u}|<K_0\}}V(x)|\widetilde{u}|^{p+\mu-1}|\widetilde{u}_k|^{\nu}dx \nonumber \\
&\leq&\|V\|_{L^{\infty}(\Omega)}\Big
(\int_{\Omega\cap\{|\widetilde{u}|\geq K_0\}}|\widetilde{u}|^{p+\mu-1}dx\Big)^{1-\frac{p}{p_s^*}}\Big
(\int_{\Omega\cap\{|\widetilde{u}|\geq K_0\}}|\widetilde{u}|^{p+\mu-1}|\widetilde{u}_k|
^{\frac{\nu p_s^*}{p}}dx\Big)^{\frac{p}{p_s^*}}\nonumber \\
&& +K_0^{\mu+\nu-1}\int_{\Omega}V(x)|\widetilde{u}|^pdx \nonumber \\
&\leq&\frac{1}{4C_0}\Big
(\int_{\Omega\cap\{|\widetilde{u}|\geq K_0\}}|\widetilde{u}|^{p+\mu-1}|\widetilde{u}_k|
^{\frac{\nu p_s^*}{p}}dx\Big)^{\frac{p}{p_s^*}}
+K_0^{\mu+\nu-1}||u||_W^p
\end{eqnarray}
and
\begin{eqnarray}\label{buuk}
&&\int_{\Omega}b(x)|\widetilde{u}|^{q+\mu-1}|\widetilde{u}_k|^{\nu}dx \nonumber \\
&=&\int_{\Omega\cap\{|\widetilde{u}|\geq K_0\}}b(x)|\widetilde{u}|^{q+\mu-1}|\widetilde{u}_k|^{\nu}dx
+\int_{\Omega\cap\{|\widetilde{u}|<K_0\}}b(x)|\widetilde{u}|^{q+\mu-1}|\widetilde{u}_k|^{\nu}dx\nonumber \\
&\leq&\|b\|_{L^{\infty}(\Omega)}\Big
(\int_{\Omega\cap\{|\widetilde{u}|\geq K_0\}}|\widetilde{u}|^{q+\mu-1}dx\Big)^{1-\frac{p}{p_s^*}}\Big
(\int_{\Omega\cap\{|\widetilde{u}|\geq K_0\}}|\widetilde{u}|^{q+\mu-1}|\widetilde{u}_k|
^{\frac{\nu p_s^*}{p}}dx\Big)^{\frac{p}{p_s^*}}\nonumber \\
&&+K_0^{\mu+\nu-1}\int_{\Omega}b(x)|\widetilde{u}|^qdx\nonumber \\
&\leq&\frac{1}{4C_0}\Big
(\int_{\Omega\cap\{|\widetilde{u}|\geq K_0\}}|\widetilde{u}|^{q+\mu-1}|\widetilde{u}_k|
^{\frac{\nu p_s^*}{p}}dx\Big)^{\frac{p}{p_s^*}}
+K_0^{\mu+\nu-1}||u||_W^q.
\end{eqnarray}
Next we consider the following two cases to check that
$f(x,\widetilde{u})\in L^{\widetilde{q}}(\Omega)$ for $\widetilde{q}>\frac{N}{ps}$.
\par Case 1\quad If $q\leq p_s^*$, we take $\mu=1+p_s^*-q$ such that $p+\mu-1<r+\mu-1<q+\mu-1\leq\frac{(\mu+p-1)p_s^*}{p}$. Combining
(\ref{uukC0}) and (\ref{auuk})-(\ref{buuk}) yields
$$\Big(\int_{\Omega}|\widetilde{u}|^{\frac{(\mu+p-1)p_s^*}{p}}|\widetilde{u}_k|^{\frac{\nu p_s^*}{p}}dx\Big)^{\frac{p}{p_s^*}}
\leq4C_0K_0^{p_s^*-q+\nu}\Big(\lambda C_1\Big(\int_{\mathbb{R}^N}\frac{{a(x)}^{\frac{1}{s_4}}}{{b(x)}^{\gamma}}\dx\Big)^{s_4}||u||_W^r+
||u||_W^p+||u||_W^q\Big)$$
and then letting $k\rightarrow\infty$ obtains
\begin{eqnarray}\label{upsp}\int_{\Omega}|\widetilde{u}|^{\frac{(\mu+\nu+p-1)p_s^*}{p}}dx<\infty\end{eqnarray}
for all $\nu>0$. Recall that $1<p<r<q\leq p_s^*$ and $f(x,\widetilde{u})=\lambda a(x)|\widetilde{u}|^{r-2}\widetilde{u}-b(x)|\widetilde{u}|^{q-2}\widetilde{u}-V(x)|\widetilde{u}|^{p-2}\widetilde{u}$.
Since $\mu+\nu+p-1=p_s^*-q+\nu+p>p$,
it can be easily verified that $f(x,\widetilde{u})\in L^{\widetilde{q}}(\Omega)$ for $\widetilde{q}>\frac{N}{ps}$
by substituting some proper values of $\nu$ into (\ref{upsp}).
\par Case 2\quad If $q>p_s^*$, we take $\mu=1+\frac{p(q-p_s^*)}{p_s^*-p}$ such that $$p+\mu-1<r+\mu-1<q+\mu-1
=\frac{p_s^*(q-p)}{p_s^*-p}=\frac{(\mu+p-1)p_s^*}{p}.$$
Using the assumption that $u\in W\cap L_{loc}^{\frac{p_s^*(q-p)}{p_s^*-p}}(\mathbb{R}^{N})$
 and repeating the argument of (\ref{geqK0})-(\ref{buuk}) we can also get the same result (\ref{upsp}) and thus $f(x,\widetilde{u})\in L^{\widetilde{q}}(\Omega)$ for $\widetilde{q}>\frac{N}{ps}$.
\par Therefore, by the arbitrariness of $\Omega$ we can deduce from \cite[Theorem 3.1]{BP} and \cite[Theorem 3.13]{BP} that
any weak solution in $W$ to problem \eqref{problem1.1} is continuous
in $\mathbb{R}^{N}$ if $q\leq p_s^*$, and that any weak solution in $W\cap L_{loc}^{\frac{p_s^*(q-p)}{p_s^*-p}}(\mathbb{R}^{N})$
to problem \eqref{problem1.1} is continuous in $\mathbb{R}^{N}$ if $q> p_s^*$, which proves Theorem \ref{continuity}.\end{proof}

\section{Appendix}
Now we correct the proofs of both \cite[Lemma 10]{PXZ1} and  \cite[Lemma A.6]{PXZ} for $1<p<2$.
\begin{lem} The set $E=(E,\|\cdot\|_E)$ is a uniformly convex Banach space.
\end{lem}
\begin{proof} Since it has been proved in \cite[Lemma 10]{PXZ1} and \cite[Lemma A.6]{PXZ} that $E$ is
a Banach space for $p>1$ and $E$ is uniformly convex when $p\geq2$, we only need to show that $E$ is
uniformly convex when $1<p<2$. Recall that
$||u||_E=\left([u]_{s,p}^p+||u||_{p,V}^p\right)^{\frac{1}{p}}$ and denote
$$W^{s,p}(\mathbb{R}^N)=\{u\in L^p(\mathbb{R}^N): \ [u]_{s,p}<\infty\},$$
$$L^p(\mathbb{R}^N,V)=\{u\in L^p(\mathbb{R}^N): \ \|u\|_{p,V}<\infty\}.$$
We first show that both $\big(W^{s,p}(\mathbb{R}^N\big),[\cdot]_{s,p})$ and $\big(L^p(\mathbb{R}^N,V),\|\cdot\|_{p,V}\big)$
are uniformly convex.
\par Since $1<p<2$, by \cite[Lemma 2.13]{AF} and \cite[Lemma 2.37]{AF} for any $u,v\in W^{s,p}(\mathbb{R}^N\big)$
we have
\begin{eqnarray}\label{uvsp}
&&\Big[\frac{u+v}{2}\Big]_{s,p}^{p^{\prime}}+\Big[\frac{u-v}{2}\Big]_{s,p}^{p^{\prime}}\nonumber\\
&=& \Big\|\Big(\Big|\frac{u(x)+v(x)}{2}-\frac{u(y)+v(y)}{2}\Big|\big|x-y\big|
^{\frac{-N-ps}{p}}\Big)^{p^{\prime}}\Big\|_{L^{p-1}(\mathbb{R}^{2N})} \nonumber\\
&&+\Big\|\Big(\Big|\frac{u(x)-v(x)}{2}-\frac{u(y)-v(y)}{2}\Big|\big|x-y\big|
^{\frac{-N-ps}{p}}\Big)^{p^{\prime}}\Big\|_{L^{p-1}(\mathbb{R}^{2N})} \nonumber\\
&\leq&\Big\|\big|x-y\big|^{\frac{-N-ps}{p-1}}\Big(\Big|\frac{u(x)+v(x)}{2}-\frac{u(y)+v(y)}{2}\Big|^{p^{\prime}}
+\Big|\frac{u(x)-v(x)}{2}-\frac{u(y)-v(y)}{2}\Big|^{p^{\prime}}\Big)\Big\|_{L^{p-1}(\mathbb{R}^{2N})}\nonumber\\
&\leq&\Big\|\big|x-y\big|^{\frac{-N-ps}{p-1}}\Big(\frac{1}{2}|u(x)-u(y)|^p+\frac{1}{2}|v(x)-v(y)|^p\Big)^{\frac{1}{p-1}}
\Big\|_{L^{p-1}(\mathbb{R}^{2N})}\nonumber\\
&=&\Big(\frac{1}{2}[u]_{s,p}^p+\frac{1}{2}[v]_{s,p}^p\Big)^{\frac{1}{p-1}}.
\end{eqnarray}
So for each $\varepsilon\in(0,2]$, by (\ref{uvsp}) there exists $\delta_1=1-\big(1-(\varepsilon/2)^{p^{\prime}}\big)^{1/p^{\prime}}>0$ such that
if $[u]_{s,p}=[v]_{s,p}=1$ and $[u-v]_{s,p}\geq\varepsilon$, then $\Big[\frac{u+v}{2}\Big]_{s,p}\leq1-\delta_1$, which implies
$\big(W^{s,p}(\mathbb{R}^N\big),[\cdot]_{s,p})$ is uniformly convex. Similarly, it can be proved that
$\big(L^p(\mathbb{R}^N,V),\|\cdot\|_{p,V}\big)$ is also uniformly convex by repeating the above argument.
\par Now for any given $\varepsilon\in(0,2]$ we take $u,v\in E$ with $\|u\|_E=\|v\|_E=1$ and $\|u-v\|_E\geq\varepsilon$, then
$$[u]_{s,p}^p+\|u\|_{p,V}^p=\|u\|_E^p=1,$$
$$[v]_{s,p}^p+\|v\|_{p,V}^p=\|v\|_E^p=1,$$
$$[u-v]_{s,p}^p+\|u-v\|_{p,V}^p=\|u-v\|_E^p\geq\varepsilon^p.$$
Without loss of generality we may assume that $[u-v]_{s,p}^p\geq\varepsilon^p/2$, that is,
$[u-v]_{s,p}\geq\varepsilon/2^{1/p}$.
Next we show by contradiction
that there exists $\delta_2>0$ depending on $\varepsilon$ such that
\begin{equation}\label{claim}\Big[\frac{u+v}{2}\Big]_{s,p}^p\leq\frac{1-\delta_2}{2}([u]_{s,p}^p+[v]_{s,p}^p).\end{equation}
\par Note that $[u]_{s,p}\leq1$ and $[v]_{s,p}\leq1$ we have two cases for the proof of (\ref{claim}) as follows.
\par Case 1\quad When $[u]_{s,p}=1$ and $[v]_{s,p}\leq1$, suppose that (\ref{claim}) is false, then there must exist
an $\varepsilon_0>0$ and two sequences $\{u_k\}$ and $\{v_k\}$ in $E$ with $[u_k]_{s,p}=1, [v_k]_{s,p}\leq1$ and
$[u_k-v_k]_{s,p}\geq\varepsilon_0/2^{1/p}$ such that
\begin{equation}\label{uk+vk}\Big[\frac{u_k+v_k}{2}\Big]_{s,p}^p\geq\frac{1}{2}\big(1-\frac{1}{k}\big)([u_k]_{s,p}^p+[v_k]_{s,p}^p).\end{equation}
We claim that $\lim\limits_{k\rightarrow\infty}[v_k]_{s,p}=1$. Otherwise, one can take a subsequence $\{v_{k_i}\}\subset\{v_k\}$ such that
$[v_{k_i}]_{s,p}\leq A<1$, and apply triangle inequality to obtain
\begin{equation}\label{uki+vki}\Big[\frac{u_{k_i}+v_{k_i}}{2}\Big]_{s,p}^p\leq\frac{1}{2^p}\big(1+[v_{k_i}]_{s,p}\big)^p
\leq\frac{[u_{k_i}]_{s,p}^p+[v_{k_i}]_{s,p}^p}{2}\cdot\Big(\frac{1+A}{2}\Big)^p\Big/\Big(\frac{1+A^p}{2}\Big),\end{equation}
which contradicts (\ref{uk+vk}) since $\Big(\frac{1+A}{2}\Big)^p\Big/\Big(\frac{1+A^p}{2}\Big)<1$, so the claim holds.
\par Denote $w_k=\frac{v_k}{[v_k]_{s,p}}$, then $\lim\limits_{k\rightarrow\infty}[v_k-w_k]_{s,p}=0$. This, together with
(\ref{uk+vk}) and $\lim\limits_{k\rightarrow\infty}[v_k]_{s,p}=1$, yields
$$1=\lim\limits_{k\rightarrow\infty}\Big[\frac{u_k+v_k}{2}\Big]_{s,p}\leq
\lim\limits_{k\rightarrow\infty}\Big[\frac{u_k+w_k}{2}\Big]_{s,p}\leq1,$$
and thus $\lim\limits_{k\rightarrow\infty}\Big[\frac{u_k+w_k}{2}\Big]_{s,p}=1$.
However, by $[u_k-v_k]_{s,p}\geq\varepsilon_0/2^{1/p}$ for all $k\geq1$ we can get an integer $k_0$ such that
$[u_k-w_k]_{s,p}\geq\frac{\varepsilon_0}{2^{1+1/p}}$ for each $k\geq k_0$, and thus $\Big[\frac{u_k+w_k}{2}\Big]_{s,p}
\leq1-\delta_3$ by definition of uniform convexity, where $\delta_3>0$ depends on $\varepsilon_0$, which is a contradiction to $\lim\limits_{k\rightarrow\infty}\Big[\frac{u_k+w_k}{2}\Big]_{s,p}=1$. Hence, (\ref{claim}) follows.
\par Case 2\quad When $[u]_{s,p}\leq1$ and $[v]_{s,p}\leq1$, without loss of generality we may assume $[u]_{s,p}\geq[v]_{s,p}>0$.
Denote $\widetilde{u}=\frac{u}{[u]_{s,p}},\ \widetilde{v}=\frac{v}{[u]_{s,p}}$, then $[\widetilde{u}\ ]_{s,p}=1,\ [\widetilde{v}\ ]_{s,p}\leq1$ and
$[\widetilde{u}-\widetilde{v}\ ]_{s,p}\geq\varepsilon/2^{1/p}$ in that $[u-v]_{s,p}\geq\varepsilon/2^{1/p}$ for the above $\varepsilon$.
By the result of Case 1 the inequality (\ref{claim}) holds for
$\widetilde{u}$ and $\widetilde{v}$, and so (\ref{claim}) follows for  $u$ and $v$.
\par Therefore, noting $\frac{[u]_{s,p}^p+[v]_{s,p}^p}{2}\geq\Big[\frac{u-v}{2}\Big]_{s,p}^p\geq\frac{\varepsilon^p}{2^{p+1}}$
and applying (\ref{claim}) with $(a+b)^p\leq2^{p-1}(a^p+b^p)$ for all $a, b>0$ and $p\geq1$ obtain
\begin{eqnarray}
\Big\|\frac{u+v}{2}\Big\|_E&=&\Big(\Big[\frac{u+v}{2}\Big]_{s,p}^p+\Big\|\frac{u+v}{2}\Big\|_{p,V}^p\Big)^{\frac{1}{p}}\nonumber \\
&\leq&\Big((1-\delta_2)\frac{[u]_{s,p}^p+[v]_{s,p}^p}{2}+\frac{\|u\|_{p,V}^p+\|v\|_{p,V}^p}{2}\Big)^{\frac{1}{p}}\nonumber \\
&\leq&\Big(1-\delta_2\frac{\varepsilon^p}{2^{p+1}}\Big)^{\frac{1}{p}}\nonumber \\
&=:&1-\delta
\end{eqnarray}
with $\delta=1-\Big(1-\delta_2\frac{\varepsilon^p}{2^{1+p}}\Big)^{1/p}>0$, which proves $E$ is
uniformly convex when $1<p<2$.
\end{proof}

\end{document}